\documentclass[reqno, 10pt, a4paper, oneside]{amsart}
\usepackage{amsmath}
\usepackage{amsfonts}
\usepackage{amssymb}
\usepackage[ansinew]{inputenc}
\usepackage{graphicx}
\usepackage{amsbsy}
\usepackage{fancyhdr}
\usepackage{yfonts}
\usepackage{hyperref}
\usepackage{eufrak}
\usepackage{enumerate}

\numberwithin{equation}{section}
\newtheorem{teo}{Theorem}[section]
\newtheorem*{inverse}{Inverse Sieve Problem}
\newtheorem{prop}[teo]{Proposition}
\newtheorem{lema}[teo]{Lemma}
\newtheorem{conj}[teo]{Conjecture}
\newtheorem{coro}[teo]{Corollary}

\theoremstyle{definition}
\newtheorem{ej}[teo]{Example}
\newtheorem{defi}[teo]{Definition}
\title{The inverse sieve problem in high dimensions}
\author{Miguel N. Walsh}
\address{Departamento de Matemática, Facultad de Ciencias Exactas y Naturales, Universidad de Buenos Aires, 1428 Buenos Aires, Argentina}
\email{mwalsh@dm.uba.ar}
\thanks{The author was partially supported by a CONICET doctoral fellowship.}

\begin{document}
\maketitle

\begin{abstract}
We show that if a big set of integer points $S \subseteq [0,N]^d$, $d>1$, occupies few residue classes mod $p$ for many primes $p$, then it must essentially lie in the solution set of some polynomial equation of low degree. This answers a question of Helfgott and Venkatesh.
\end{abstract}

\bigskip

\section{Introduction}

One of the main topics of study in analytic number theory is the distribution of sets of integers in residue classes. Examples abound, but folkloric ones include Dirichlet's theorem, which tells us that the primes are uniformly distributed along primitive residue classes, and the open problem of determining how large may the least quadratic non-residue be.

\medskip
On the other hand, in the expanding subject of arithmetic combinatorics, much of the focus has been in establishing what is known as {\sl inverse theorems} in which one starts with a set having a specific arithmetic property and wishes to use this information to give a characterization of the set. Notable examples of this include Freiman type theorems (see for instance \cite{GR,Helfgott,Tao} and the survey \cite{Green survey}), inverse theorems for the Gowers norm \cite{GTZ} and the inverse Littlewood-Offord theory \cite{TV book,TV}.

\medskip
This paper is concerned with the problem of connecting both lines of inquiry by establishing an inverse theorem for the distribution of sets in residue classes.  Since we would expect a random set to be fairly well distributed, the main question here is whether a set occupying very few residue classes for many primes $p$ has to have some specific structure.  The remarkable observation that this might indeed be the case is due to Croot and Elsholtz \cite{CL} and Helfgott and Venkatesh \cite{HV}. Writing $[N]$ for the set of integers $\left\{0,\ldots,N \right\}$ their observation can be resumed in the following principle:

\begin{inverse}
Suppose a set $S \subseteq [N]^d$ occupies very few residue classes mod $p$ for many primes $p$. Then, either $S$ is small, or it possesses some strong algebraic structure.
\end{inverse}

There is a good reason why such inverse sieve results are of much interest in number theory. One of the main features of sieve theory is the uniformity of its results, which is a consequence of the fact that sieves only take into account the cardinality of the classes occupied by the set. However, a clear drawback of this is that the bounds thus obtained are limited to what happens in extremal cases. By stating that such extremal sets must have a very specific structure, inverse results should allow one to retain the uniformity of the sieve while providing much stronger bounds. The reader may consult the book of Kowalski \cite[\S 2.5]{Kowalski book} for further discussion of the potential applications of this phenomenon and \cite{Green survey} for applications of similar classifications in arithmetic combinatorics.

\medskip
In this paper we give a satisfactory answer to the inverse sieve problem for every $d \ge 2$. In order to discuss our results suppose we are given a big integer set $S \subseteq [N]^d$ occupying $O(p^{d-1})$ residue classes in $(\mathbb{Z}/p \mathbb{Z})^d$ for many primes $p$. What does this imply about $S$? By the Lang-Weil inequality, we know that this condition is satisfied by the set of integer points of a proper algebraic variety of small degree and one would expect a partial converse to also hold. That is, that any big set $S \subseteq [N]^d$ occupying only that many residue classes for every prime $p$ should essentially be contained inside the solution set of a polynomial of low degree. When $d=1$ this follows from Gallagher's larger sieve \cite{Gallagher} (not to be confused with the conjecture discussed in \S \ref{d=1}). The case $d=2$ was proven by Helfgott and Venkatesh in \cite{HV}, by applying the Bombieri-Pila determinant method \cite{BP} to obtain a two-dimensional generalization of the larger sieve. Although their methods are only capable of handling the case $d \le 2$, they conjectured that such an inverse theorem should in fact hold for every dimension $d$. In this paper we introduce a different approach and use it to answer this question by giving the following best possible result.

\begin{teo}
\label{principal}
Let $0 \le k < d$ be integers and let $\varepsilon,\alpha, \eta > 0$ be positive real numbers. Then, there exists a constant $C$ depending only on the above parameters, such that for any set $S \subseteq [N]^d$ occupying less than $\alpha p^k$ residue classes for every prime $p$ at least one of the following holds: 
\begin{enumerate}[(i)]
	\item ($S$ is small) $\left|S \right| \ll_{d,k,\varepsilon,\alpha} N^{k-1+\varepsilon}$,
	\item ($S$ is strongly algebraic) There exists a polynomial $f \in \mathbb{Z}[x_1, \ldots, x_d]$ of degree at most $C$ and coefficients bounded by $N^C$ vanishing at more than $(1-\eta) |S|$ points of $S$.
\end{enumerate}
\end{teo}

Theorem \ref{principal} is sharp. Indeed, the reader may consult Section \S \ref{thin} for examples of sets of size $|S| \gg N^{k-1}$ occupying less than $p^k$ residue classes for every prime $p$ but possessing no algebraic structure. On the other hand, we only need to require from $S$ that it occupies few residue classes for sufficiently many small primes (see Theorem \ref{refinement}). More generally, we will show in Theorem \ref{regular teo} that assuming some necessary regularity conditions, every set of size $\gg N^{\varepsilon}$ occupying few residue classes for many primes $p$ must satisfy condition (ii). In Section \S \ref{Approximate reduction} we shall give an easy application of this generalization to the characterization of functions preserving some structure when reduced to prime moduli.

\medskip
Taking $d=2$ in Theorem \ref{principal} we recover the result of \cite{HV}. Actually, the methods of Helfgott and Venkatesh are capable of handling the case $k=1$ of Theorem \ref{principal}, that is, when $S$ is assumed to occupy only $O(p)$ residue classes. However, the approach fails as soon as the set occupies more than $p \log p$  classes. The reason for this is that their method, as well as the larger sieve itself, is in essence a counting argument (see \S \ref{gallagher sieve}) and therefore needs the \emph{number} of classes occupied by $S$ to be small, while the high dimensional setting requires us to take advantage of the local \emph{density} of $S$ being small. This type of obstacle is not specific to the problem at hand, but arises whenever one tries to extend this kind of sieves to higher dimensional settings (see \cite[Remark 3]{Kowalski} for some discussion). So while we do make use of the larger sieve, in order to establish Theorem \ref{principal} we need to introduce an approach that overcomes this difficulty by taking advantage of the structure of the set and which we believe to be applicable in more general situations.

\medskip
The rest of the paper is organized as follows. After setting up some notation, in Section \S \ref{cond} we state and discuss Proposition \ref{control}, which is the main ingredient of the paper, and use it to deduce Theorem \ref{principal}. This proposition says that while a polynomial identity of degree $r$ usually requires $O(r^d)$ points to be tested, if one is interested in sets satisfying hypothesis similar to those of Theorem \ref{principal} then one can verify such identities for a positive proportion of the set using only $O(r^k)$ points. Then, in Section \S \ref{applying ls}, we review some facts about the larger sieve and apply them to obtain a key uniformization lemma. Using this, the proof of Proposition \ref{control} is carried out in Section \S \ref{proof control}. Finally, in \S \ref{thin} we construct several examples showing that our results are sharp, while in \S \ref{Further} we discuss further consequences of our methods as well as the remaining case ($d=1$) of the inverse sieve problem.

\medskip
\noindent \emph{Acknowledgment.} The author would like to thank his advisor Román Sasyk for several comments and suggestions during the preparation of the paper.

\section{A conditional proof of Theorem 1.1}
\label{cond}

\subsection{General notation}
\label{notation}

We now fix some notation. By $O_{c_1, \ldots, c_k}(X)$ we shall mean a quantity which is bounded by $C_{c_1,\ldots,c_k}X$ where $C_{c_1,\ldots, c_k}$ is some finite positive constant depending on $c_1, \ldots, c_k$. Also, we shall write $Y \ll_{c_1, \ldots, c_k} X$ to mean $|Y|=O_{c_1,\ldots,c_k}(X)$. However, since we will generally be concerned with the study of a set $S$ satisfying the hypothesis of Proposition \ref{control} for some parameters $d,h,\kappa$ and $\varepsilon$ as in the statement of that proposition, we will free up some notation by assuming that all implied constants in the $O,\ll$ notation \emph{always} depend on these parameters even though this may not be explicitly stated. So for instance $Y \ll_{\eta} X$ stands for $Y \ll_{\eta,d,h,\kappa,\varepsilon} X$. Throughout the paper we will let the letter $c$ denote a small positive constant whose exact value may vary at each occurrence.

\medskip
Given a statement $\phi(x)$ with respect to an element $x \in [N]^d$ we will write ${\bf 1}_{\phi(x)}$ for the function which equals $1$ if $\phi(x)$ is true and $0$ otherwise. Also, we shall write $\pi_i:\mathbb{Z}^d \rightarrow \mathbb{Z}$, $1 \le i \le d$, for the projection to the $i$th coordinate.

\medskip
The letter $p$ will always refer to a prime number. We write $\mathcal{P}$ for the set of primes and given any magnitude $Q$, we denote $\mathcal{P}(Q)$ the set of primes $p \le Q$. Since we will usually need to consider the weight $\frac{\log p}{p}$ over $\mathcal{P}$, for a finite subset $P \subseteq \mathcal{P}$ we write $w(P):= \sum_{p \in P} \frac{\log p}{p}$. We shall use the estimates $w(\mathcal{P}(Q))=\log Q + O(1)$ and $\sum_{p \in \mathcal{P}(Q)} \log p \sim Q$ without explicit mention.

\subsection{Characteristic sets}
\label{conditional}

The purpose of this section is to state Proposition \ref{control} which is the key ingredient of the paper and use it to derive Theorem \ref{principal}. What this Proposition essentially says, is that for any ill-distributed set $S$ as in the statement of Theorem \ref{principal}, one may find a very small ``characteristic" subset $A \subseteq S$ such that if a small polynomial vanishes at $A$ then it also vanishes at a positive proportion of $S$. The task of proving Theorem \ref{principal} is thus reduced to that of finding a polynomial which vanishes at $A$, and this will always be possible since $A$ is small.

\medskip
Before proceeding, we need to define exactly what we mean for a polynomial to be small. Given a parameter $N$ and some integer $d>0$ by an $r$-polynomial, for a positive integer $r$, we shall mean any polynomial $f$ with integer coefficients satisfying $|f(n)|<N^{3r}$ for every $n \in [N]^d$. The exponent $3r$ is chosen in order to guarantee that if $N$ is sufficiently large in terms of $r$ and $d$, then a polynomial $f \in \mathbb{Z}[x_1,\ldots,x_d]$ of degree at most $r$, with coefficients bounded in absolute value by $N^r$, is an $r$-polynomial. This leads us to the following definition.

\begin{defi}
Let $0 < \delta \le 1$ be a positive real number and $r>0$ some integer. We say a subset $A$ of a set $S$ is $(r,\delta)$-characteristic for $S$ if we can find some subset $A \subseteq B \subseteq S$ of size $|B| \ge \delta |S|$ such that whenever an $r$-polynomial vanishes at $A$, then it also vanishes at $B$.
\end{defi}

We can now state Proposition \ref{control} which says that ill-distributed sets always admit characteristic subsets.

\begin{prop}
\label{control}
Let $d,h \ge 1$ be arbitrary integers and $\varepsilon > 0$ some positive real number. Set $Q=N^{\frac{\varepsilon}{2d}}$ and let $P \subseteq \mathcal{P}(Q)$ satisfy $w(P) \ge \kappa \log Q$ for some $\kappa>0$. Also, let $r$ be an arbitrary positive integer. Suppose $S \subseteq [N]^d$ is a set of size $|S| \gg N^{d-h-1+\varepsilon}$ occupying at most $\alpha p^{d-h}$ residue classes mod $p$ for every prime $p \in P$ and some $\alpha > 0$. Then, if $N$ is sufficiently large, there exists a set $A \subseteq S$ of size $|A|=O(r^{d-h})$ which is $(r, \delta)$-characteristic for $S$, for some $\delta \gg 1$ independent of $S$ or $r$.
\end{prop}

\noindent \emph{Remarks.} The exact value of $Q$ in the above statement is irrelevant and may be replaced by any small power of $N$. The reason why we have made the change of variables $h:=d-k$ with respect to Theorem \ref{principal} is that in the arguments to follow we shall always set the quantity $h$ to be fixed and induct on $d$. We believe it is simpler to introduce this change of notation at an early stage.

\medskip
To see why such a result might be expected consider some polynomial $f$ vanishing at an integer point $x$. Since polynomials descend to congruence classes, this means that for any other integer $y$ satisfying $y \equiv x (\text{mod }p)$ for a prime $p$, we will have $p|f(y)$. Thus, if we are given a set $S$ which occupies very few residue classes, one may then hope to find a small subset $A$ such that given some $y \in S$ there are a lot of primes $p$ for which $y \equiv x (\text{mod }p)$ for some $x \in A$. It would then follow that if a polynomial vanishes at $A$ then there would be many primes $p$ dividing $f(y)$. If furthermore $f$ is small, then this can  only hold if $f(y)=0$.

\medskip
On the other hand, the size hypothesis on $S$ is necessary. For instance, one may construct small (logarithmic size) sets $S \subseteq [N]$ as in \cite[\S 4.3]{HV} which occupy few residue classes for large moduli just because they are small, but which however have at most one element in each residue class, making the above argument unviable in this situation. Furthermore, it is clear that a similar pathology occurs in higher dimensions, by considering for instance the product set $S \times [N]$. For the general construction of this type of sets and to see that in fact one cannot take $\varepsilon=0$ in Proposition \ref{control} the reader is referred to \S \ref{thin}.

\medskip
In order to deduce Theorem \ref{principal} from Proposition \ref{control} we will need to find a polynomial which vanishes at a specific set of points. This will be accomplished in a standard way by means of Siegel's lemma.

\begin{lema}[Siegel]
Suppose we are given a system of $m$ linear equations
$$ \sum_{j=1}^n a_{ij} \beta_j =0 \text{   } \forall 1 \le i \le m,$$ 
 in $n$ unknowns $(\beta_1, \ldots, \beta_n)$, $n>m$, where the coefficients $(a_{ij})$ are integers not all equal to $0$ and bounded in magnitude by some constant $C$. Then, the above system has a non trivial integer solution $(\beta_1, \ldots, \beta_n)$ with $|\beta_j| \le 1+(Cn)^{m/(n-m)}$ for all $1 \le j \le n$. 
\end{lema}

\begin{proof}[Proof of Theorem \ref{principal} assuming Proposition \ref{control}]
Let the hypothesis be as in the statement of Theorem \ref{principal} and write $h := d-k$. Assume condition (i) fails, so that $|S| \gg N^{d-h-1+\varepsilon}$. We claim that for any given integer $r$ there exists a set $A \subseteq S$ of size $|A|=O_{\eta}(r^{d-h})$ which is $(r,1-\eta)$-characteristic for $S$, provided $N$ is sufficiently large. To see this we begin by noticing that Proposition \ref{control} implies the existence of some $\delta \gg 1$ such that for every subset $S' \subseteq S$ with $|S'| \ge \eta |S|$ there exists a set $A' \subseteq S'$ of size $|A'| = O_{\eta}(r^{d-h})$ which is $(r,\delta)$-characteristic for $S'$. From now on we fix this value of $\delta$. Let $A_0$ be such a characteristic subset for $S$ and let $B_0$ consist of those elements of $S$ which vanish at every $r$-polynomial that vanishes at $A_0$, so in particular $|B_0| \ge \delta |S|$. If $\delta \ge 1-\eta$ we are done, otherwise we have that $S_1 := S \setminus B_0$ satisfies $|S_1| \ge \eta |S|$ and therefore contains a characteristic subset $A_1 \subseteq S_1$ as above. If we now let $B_1$ denote those points of $S_1$ vanishing at every $r$-polynomial that vanishes at $A_1$ we see that either we get the claim with $A = A_0 \cup A_1$ or the set $S_2 := S_1 \setminus B_1$ satisfies $\eta |S| \le |S_2| \le (1-\delta)^2 |S|$. After iterating this process $j$ times we see that if the set $A= \bigcup_{i=0}^{j-1} A_i$ is not $(r,1-\eta)$-characteristic for $S$ then we can find some $S_j \subseteq S$ with $\eta |S| \le S_j \le (1-\delta)^j|S|$. Since this last possibility cannot hold for some large $j=O_{\eta}(1)$, the claim follows.

\medskip
Now it only remains to find some $r$-polynomial $f$ which vanishes at $A$ and which is of the form given in Theorem \ref{principal}. We may assume $d|r$. We thus consider the system of $|A|$ linear equations in $\left( \frac{r}{d}+1 \right)^d$ unknowns given by
\begin{equation}
\label{sistema}
\sum_{{\bf i}=\left\{ i_1, \ldots, i_d \right\} \le r/d} \beta_{\bf i} a^{{\bf i}}=0 \text{   }\forall a \in A,
\end{equation}
where ${\bf i}\le l$ stands for $i_j \le l$ for all $1 \le j \le d$ and where we use the multi-index notation $a^{\bf i} = a_1^{i_1} \ldots a_d^{i_d}$ for $a=(a_1,\ldots,a_d)$. Notice that $|a^{\bf i}| \le N^r$. If we now choose $r=O_{\eta}(1)$ large enough so that $\left( \frac{r}{d}+1 \right)^d > 3|A|$ it follows by Siegel's lemma that there exists an integer solution $(\beta_{\bf i})$ to (\ref{sistema}) with $|\beta_{\bf i}| \ll_r N^{r/2} \le N^r$ provided $N$ is sufficiently large. We thus see that the polynomial $f:= \sum_{{\bf i} \le r/d} \beta_{\bf i}x^{\bf i}$ is of the desired form (assuming again that $N$ is sufficiently large) and, taking $C=r$, this concludes the proof of Theorem \ref{principal}.
\end{proof}

Notice that we have actually proved the following slight strengthening of Theorem \ref{principal} in which the set $S$ is only required to be badly distributed in a dense subset of the primes.

\begin{teo}
\label{refinement}
Let $0 \le k < d$ be integers and let $\varepsilon, \eta > 0$ be positive real numbers. Set $Q=N^{\frac{\varepsilon}{2d}}$ and let $P \subseteq \mathcal{P}(Q)$ satisfy $w(P) \ge \kappa \log Q$ for some $\kappa>0$. Suppose $S \subseteq [N]^d$ is a set of size $|S| \gg N^{k-1+\varepsilon}$ occupying at most $\alpha p^{k}$ residue classes mod $p$ for every prime $p \in P$ and some $\alpha > 0$. Then there exists a polynomial $f \in \mathbb{Z}[x_1, \ldots, x_d]$ of degree $O_{\eta}(1)$ and coefficients bounded by $N^{O_{\eta}(1)}$ which vanishes at more than $(1-\eta) |S|$ points of $S$.
\end{teo}

\noindent \emph{Remark.} Since we have already mentioned that the exact value of $Q$ in Proposition \ref{control} is irrelevant, it follows that Theorem \ref{refinement} also holds with $Q$ any small power of $N$.

\section{Applying the larger sieve in high dimensions}
\label{applying ls}

\subsection{A review of the larger sieve}
\label{gallagher sieve}

We will now quickly review some facts about Gallagher's larger sieve and use them to prove two easy lemmas which we shall need later. For further discussion of the larger sieve and its consequences the reader may consult \cite[Section 2.2]{CM} and of course Gallagher's original paper \cite{Gallagher}.

\medskip
Before proceeding we need to state some further notation that will be used in this and the next sections. When studying a set $S \subseteq [N]^d$ we will denote by $[S]_p$ the set of residue classes mod $p$ occupied by $S$. Given such a set $S$, we shall be largely concerned with how many elements of $S$ belong to a given residue class, so it is important for us to have a specific notation for this subset. Thus, given a residue class ${\bf a}=(a_1,\ldots, a_d) (\text{mod }p)$ we write $S({\bf a};p)$ to refer to those elements of $S$ which are congruent to ${\bf a} (\text{mod }p)$. Moreover, we shall sometimes consider some $a \in \mathbb{Z}/p \mathbb{Z}$ and write $S(a;p)$ for those elements of $S$ having their first coordinate congruent to $a$ (mod $p$). Since we will always use the bold font ${\bf a}$ to denote a vector residue class and since where this class lives shall be clear from the context altogether, we believe the similarity of both notations will not cause any confusion. Finally, if $p$ is fixed, we may simply write $S({\bf a})$ and $S(a)$ for the above sets.

\medskip
Fix now a set $S$ and consider some parameter $Q$. The main idea of the larger sieve is to count in two different ways the number of distinct pairs $x,y \in S$ and primes $p \le Q$ such that $x \equiv y (\text{mod }p)$. Given two such integers $x,y \in [N]$ it is clear that those primes for which they are congruent are exactly those dividing $|x-y| \le N$ and therefore
\begin{equation}
\label{Gallagher}
 \sum_{p \le Q} \sum_{\substack{x,y \in S \\ x \neq y}}{\bf 1}_{x \equiv y (\text{mod }p)} \log p \le |S|^2 \log N.
 \end{equation}
On the other hand, we have that the left hand side of (\ref{Gallagher}) equals
\begin{equation}
\label{Gallagher2}
\sum_{p \le Q} \sum_{a (\text{mod }p)} |S(a;p)|^2 \log p - |S| \sum_{p \le Q} \log p.
\end{equation}
Notice that the above argument also works if $S \subseteq [N]^d$ since if $p$ is a prime for which $x \equiv y (\text{mod }p)$ then $p$ must divide $|\pi_1(x)-\pi_1(y)|$ which is bounded by $N$.

\medskip
As an example, we have the following result due to Gallagher \cite{Gallagher}. Suppose we are given a set $S \subseteq [N]$ occupying at most $\alpha p$ residue classes, then the Cauchy-Schwarz inequality implies
$$ \sum_{a (\text{mod }p)} |S(a;p)|^2 \ge \frac{1}{\alpha p}|S|^2.$$
Combining this with (\ref{Gallagher}) and (\ref{Gallagher2}) we obtain
$$ \frac{1}{\alpha} \log Q+O \left( \frac{|Q|}{|S|} \right) \le \log N+O(1).$$
Taking $Q=|S|$ we conclude that $|S| \ll_{\alpha} N^{\alpha}$.

\medskip
For the purposes of this paper, we need to apply Gallagher's sieve in a slightly more general context. Precisely, we will use the following lemma.

\begin{lema}
\label{larger sieve}
Let $X \subseteq [N]$ be some set of integers and set $Q=N^{\gamma}$ for some $\gamma > 0$. Let $c_1,c_2 > 0$ be positive real numbers. Suppose there is a set of primes $P \subseteq \mathcal{P}(Q)$ with $w(P) \ge c_1 \log Q$ such that for every $p \in P$ there are at least $c_2 |X|$ elements of $X$ lying in at most $\alpha p$ residue classes for some $\alpha > 0$ independent of $p$. Then, if $\alpha$ is sufficiently small in terms of $c_1,c_2$ and $\gamma$, it must be $|X|<Q$. 
\end{lema}

\begin{proof}
Again, we count the number of pairs $x,y \in X$ and $p \in P$ with $x \equiv y (\text{mod }p)$. On one hand, we have as before that
\begin{equation}
\label{gala}
\sum_{p \in P} \sum_{\substack{x,y \in X \\ x \neq y}}{\bf 1}_{x \equiv y (\text{mod }p)} \log p \le |X|^2 \log N.
\end{equation}
On the other hand, using the Cauchy-Schwarz inequality we see that our hypothesis on $X$ implies
$$ \sum_{a (\text{mod }p)} |X(a;p)|^2 \ge \frac{1}{\alpha p}(c_2|X|)^2,$$
from where it follows that the left hand side of (\ref{gala}) is at least 
$$\frac{c_1 c_2^2}{\alpha}|X|^2 \log Q +O ( |Q||X|).$$
It is then clear that if $\alpha$ is sufficiently small, then the only way for (\ref{gala}) to hold is to have $|X|<Q$.
\end{proof}

Finally, we prove the following easy consequence of the larger sieve which already handles the case $d=h$ of Proposition \ref{control}. 

\begin{lema}
\label{finito}
Let $Q=N^{\gamma}$ for some $\gamma > 0$ and let $P \subseteq \mathcal{P}(Q)$ be some set of primes with $w(P) \ge c_1 \log Q$ for some $c_1>0$. Let $S \subseteq [N]^d$ occupy less than $c_2$ residue classes mod $p$ for every prime $p \in P$ and some constant $c_2$. Then $|S|= O_{c_1,c_2,\gamma}(1)$.
\end{lema}

\begin{proof}
Gallagher's sieve implies in this case
$$ \log N \ge \left( \frac{1}{c_2} - \frac{1}{|S|} \right)\sum_{p \in P} \log p \gg \left( \frac{1}{c_2} - \frac{1}{|S|} \right) N^{\gamma c_1},$$
and clearly, for sufficiently large $N$, this can only hold if $|S| \le c_2$.
\end{proof}

\subsection{Genericity}
\label{generic sets}

Our strategy to prove Proposition \ref{control} will be to partition $S$ into many lower dimensional subsets and apply induction. However, the main obstacle we encounter in doing so (and which is not merely a technical issue, as can be seen from the examples in \S \ref{thin}) is the possibility that the resulting subsets are rather independent from each other, in the sense that they do not share many residue classes. If this happens, then the fact that a small polynomial vanishes at one of this subsets will not give us much information about the behavior of this polynomial in the other subsets. However, in order for this to happen it would be necessary for these subsets to occupy very few residue classes and this would imply the existence of too many elements in each subset occupying the same residue class. While with our hypothesis one cannot guarantee that this never happens, the goal of this section is to show that this indeed does not happen on average, which will be sufficient for our arguments.

\medskip
We begin with the following definition.

\begin{defi}[Genericity]
Given a real number $B>0$ and some integer $l>0$ we say that a set $S \subseteq [N]^d$ is $(B,l)$-generic mod $p$ if
$$\frac{|S({\bf a};p)|}{|S|} < \frac{B}{p^l},$$
for every residue class ${\bf a} (\text{mod }p)$.
\end{defi} 

Given a set of primes $P \subseteq \mathcal{P}(Q)$ we shall write $P' \hookrightarrow P$ to mean a subset $P' \subseteq P$ with $w(P') \gg w(P)$. Recall that by our conventions in \S \ref{notation} the implied constants depend on the parameters $d,h,\varepsilon,\kappa$ of Proposition \ref{control}. The rest of this section is devoted to the proof of the following lemma.

\begin{lema}
\label{genericity}
Let $d,h \ge 1$ be arbitrary integers and let $\varepsilon > 0$ be some positive real number. Set $Q=N^{\frac{\varepsilon}{2d}}$ and let $P \subseteq \mathcal{P}(Q)$ satisfy $w(P) \ge \kappa \log Q$ for some $\kappa>0$. Suppose $S \subseteq [N]^d$ is a set of size $|S| \gg N^{d-h-1+\varepsilon}$ occupying at most $\alpha p^{d-h}$ residue classes mod $p$ for every prime $p \in P$ and some $\alpha > 0$. Then there exists $B=O(1)$ and a set of primes $P' \hookrightarrow P$ such that for each $p \in P'$ there is some subset $\mathcal{G}_p(S) \subseteq S$, $|\mathcal{G}_p(S)| \gg |S|$, which is $(B,d-h)$-generic mod $p$.
\end{lema}

{\noindent \emph{Remarks.}} Here again the exact value of $Q$ is not important as long as it is a small power of $N$. Also, as in the previous statements, all the hypothesis are necessary because of the examples in \S \ref{thin}.

\begin{proof}
From now on fix an integer $h \ge 1$. If $d<h$ the result is trivial, while for $d=h$ it follows from Lemma \ref{finito} with $B=|S|$ and $P'=P$. We will proceed by induction on $d$. Thus, let $d \ge h+1$ be some integer and assume the result holds for every smaller dimension.

\medskip

Take $S$ and $P$ as in the statement and recall that $\pi_i(S)$ is the projection of $S$ to the $i$th coordinate. We claim that for some $1 \le i \le d$ there exists a set $S' \subseteq S$ with $|S'| \ge |S|/2^d$ such that every $A \subseteq S'$ with $|A| \ge |S'|/2$ satisfies $|\pi_i(A)| \ge Q$. Indeed, if the claim fails with $S' =S$ and $i=1$ we may find some subset $S_1 \subseteq S$ with $|S_1| \ge |S|/2$ and $|\pi_1(S_1)|<Q$. Then, if the claim fails again with $S'=S_1$ and $i=2$, we get some $S_2 \subseteq S_1$ with $|S_2| \ge |S_1|/2 \ge |S|/4$ and $|\pi_1(S_2)|,|\pi_2(S_2)|<Q$. Iterating this $d$ times either we get the claim or end up with a set $S_d \subseteq S$ satisfying
$$ |S| \le 2^d |S_d| \le 2^d |\pi_1(S_d)| \ldots |\pi_d(S_d)| < 2^d Q^d.$$
By our choice of $Q$ this is clearly absurd for sufficiently large $N$ and therefore the claim follows. 

\medskip
Since it sufficies to prove the lemma for such a subset $S'$ we may assume without lost of generality that $S'=S$ and permuting the coordinates if necessary we may also assume $i=1$. Hence, we have that 
\begin{equation}
\label{assumption}
|\pi_1(A)| \ge Q \text{ for every }A \subseteq S \text{ with }|A| \ge |S|/2.
\end{equation} 

\medskip
We wish to construct a dense subset of $S$ which is in an adequate position to apply the induction hypothesis. Since we will be working with the first coordinate, given some $a \in \mathbb{Z}/p \mathbb{Z}$, we will write $S(a;p)$ to refer to those elements of $S$ having their first coordinate congruent to $a(\text{mod }p)$. Let $B_1$ be some large constant to be specified later. Since $ \left|[S]_p \right| \le \alpha p^{d-h}$, it is clear that there can be at most $\alpha p /B_1$ residue classes $a \in [\pi_1(S)]_p \subseteq \mathbb{Z}/p \mathbb{Z}$ for which $\left|[S(a;p)]_p \right| \ge B_1 p^{d-h-1}$. We denote by $\mathcal{E}_1(p)$ this exceptional set. Also, we write
$$ \mathcal{E}_2(p) := \left\{ a \in [\pi_1(S)]_p :  |S(a;p)| \ge \frac{B_1}{\alpha p}|S| \right\}.$$
From the obvious fact that $\sum_{a \in \mathbb{Z}/p \mathbb{Z}}|S(a;p)| = |S|$ it follows that $|\mathcal{E}_2(p)| \le \alpha p/B_1$ and therefore $|\mathcal{E}(p)| \le 2 \alpha p / B_1$, where $\mathcal{E}(p) := \mathcal{E}_1(p) \cup \mathcal{E}_2(p)$. By means of the larger sieve we may now deduce that not too many integers in $[N]$ can lie in $\mathcal{E}(p)$ for many $p \in P$. Indeed, consider the set $X$ which consists of all elements $x \in [N]$ for which
$$ \sum_{p \in P} {\bf 1}_{x(\text{mod }p) \in \mathcal{E}(p)} \frac{\log p}{p} \ge \frac{1}{2} w(P).$$
By the pigeonhole principle, one may then find a set of primes $P_1 \subseteq P$ with $w(P_1) \ge \frac{1}{4} w(P)$ and such that $|\bigcup_{a \in \mathcal{E}(p)} X(a;p)| \ge \frac{1}{4} |X|$ for every $p \in P_1$. It then follows from Lemma \ref{larger sieve} that upon choosing $B_1$ sufficiently large, we can ensure that $|X|<Q$.

\medskip
By (\ref{assumption}), we deduce that $\left| S \setminus \pi_1^{-1}(X) \right| \ge \frac{1}{2}|S|$. We may therefore find a subset $S' \subseteq S$ with $|S'| \ge \frac{1}{4} |S|$ which does not intersect $\pi_1^{-1}(X)$ and such that $S_x' := \pi_1^{-1}(x) \cap S'$ satisfies $|S_x'| \gg N^{d-h-2+\varepsilon}$ for every $x \in \pi_1(S')$. Every such $x$ lies outside of $X$ and therefore has associated a set of primes $P_x \hookrightarrow P$ for which $x(\text{mod }p) \notin \mathcal{E}(p)$. Since $\mathcal{E}_1(p) \subseteq \mathcal{E}(p)$, we may apply the induction hypothesis to $S_x'$ for every $x$ to see that there exists sets of primes $P_x' \hookrightarrow P_x$ and constants $c,B_2>0$ independent of $x$, such that for each $p \in P_x'$ there is a $(B_2,d-h-1)$-generic mod $p$ set $\mathcal{G}_p(S_x') \subseteq S_x'$ containing at least $c \left| S_x' \right|$ elements.

\medskip
Since the sets $P_x'$ constructed above satisfy $P_x' \hookrightarrow P$, with the implied constant independent of $x$, we may apply again the pigeonhole principle to locate some set of primes $P' \hookrightarrow P$ and some constant $c>0$, such that for each $p \in P'$ there are at least $c|S'|$ elements $s \in S'$ for which $p \in P_{\pi_1(s)}'$. It thus follows that if for a prime $p \in P'$ we consider the set
$$ \mathcal{G}_p(S) := \bigcup_{x: p \in P_x'} \mathcal{G}_p \left(S_x' \right),$$
then $|\mathcal{G}_p(S)| \gg |S'| \gg |S|$ and $\mathcal{G}_p(S) \cap \pi_1^{-1}(x)=\mathcal{G}_p(S_x')$ is a $(B_2,d-h-1)$-generic set for every $x \in \pi_1(\mathcal{G}_p(S))$. Also, we see that there are at most $\frac{B_1}{\alpha p}|S| \ll \frac{B_1}{\alpha p}|\mathcal{G}_p(S)|$ elements of $\mathcal{G}_p(S)$ having the same first coordinate mod $p$ since by construction it does not lie in $\mathcal{E}_2(p)$. It thus follows that $\mathcal{G}_p(S)$ is a $B$-generic set for some large $B$ depending on $B_1$ and $B_2$ but independent of $p$ and this concludes the proof of Lemma \ref{genericity}.
\end{proof}

\section{The proof of Proposition 2.2}
\label{proof control}

In this section we give a proof of Proposition \ref{control}. As we did in the proof of Lemma \ref{genericity} we will fix an integer $h$ and induct on $d$. Since for $d \le h$ the result is either trivial or follows from Lemma \ref{finito} we may assume $d \ge h+1$ and that the result holds for all smaller dimensions.

\medskip
We are thus given a set $S$ and some positive integer $r$. Our first step will be to find generic sets inside the sections of $S$ for many primes $p$. Proceeding as in the beginning of the proof of Lemma \ref{genericity} we may assume that
\begin{equation}
\label{assumption2}
|\pi_1(A)| \ge Q \text{ for every }A \subseteq S \text{ with }|A| \ge |S|/2.
\end{equation} 
This allows us, at the cost of passing to a subset of half density if necessary, to get the bound 
\begin{equation}
\label{assumption3}
|S_x| \le 2|S|/Q \text{ for every }x \in [N],
\end{equation}
 where $S_x := \pi_1^{-1}(x) \cap S$. Finally we may also assume, again by passing to a subset of half density if necessary,  that $|S_x| \gg N^{d-h-2+\varepsilon}$ for every $x \in \pi_1(S)$.

\medskip
Let $B$ be some large constant. For every prime $p$ we denote by $\mathcal{E}(p)$ the set of residue classes $a \in \mathbb{Z}/p \mathbb{Z}$ for which $|[S(a;p)]_p| \ge B p^{d-h-1}$ (recall that $S(a;p)$ stands for those elements of $S$ having their first coordinate congruent to $a(\text{mod }p)$ and thus $[S(a;p)]_p$ consists of those residue classes in $[S]_p \subseteq (\mathbb{Z}/p \mathbb{Z})^d$ having $a$ as a first coordinate). Since $|\mathcal{E}(p)| \le \alpha p/B$, applying Lemma \ref{larger sieve} as in the proof Lemma \ref{genericity}, we conclude by (\ref{assumption2}) that if $B$ is chosen sufficiently large, we may find some $S' \subseteq S$, $|S'| \gg |S|$, such that for each $x \in \pi_1 (S')$ we have $P_x \hookrightarrow P$, with the implied constant independent of $x$, and where
$$ P_x := \left\{ p \in P : x(\text{mod }p) \notin \mathcal{E}(p) \right\}.$$
This places us in a position in which we can apply the induction hypothesis to each section $S_x'$ of $S'$ to find some $\delta_0 \gg 1$ independent of $x$ such that each $S_x'$ admits a $(r,\delta_0)$-characteristic subset of size $O(r^{d-h-1})$. In particular, we see that at the cost of passing to a subset of $S'$ of density $\delta_0$ if necessary, we may assume that inside each $S_x'$ we can find a set of size $O(r^{d-h-1})$ which is $(r,1)$-characteristic for the whole section. Notice that since we are refining the sections, we still get a bound of the form $|S_x'| \gg N^{d-h-2+\varepsilon}$ for every $x \in \pi_1(S')$. Thus, we may also apply Lemma \ref{genericity} to every such $S_x'$ obtaining sets of primes $P_x' \hookrightarrow P_x$ such that for every $p \in P_x'$ we can find a $(C,d-h-1)$-generic subset $\mathcal{G}_p(S_x) \subseteq S_x'$, $|\mathcal{G}_p(S_x)| \gg |S_x'|$, where $C$ and the implied constants are independent of $p$ and $x$. In particular, we may find some set of primes $P' \hookrightarrow P$ such that for each $p \in P'$ the set 
$$ \mathcal{G}_p(S) := \bigcup_{x: p \in P_x'} \mathcal{G}_p \left(S_x \right),$$
satisfies $|\mathcal{G}_p(S)| \gg |S'| \gg |S|$ and each nonempty section $(\mathcal{G}_p(S))_x$ of $\mathcal{G}_p(S)$ is a $(C,d-h-1)$-generic set.

\medskip
From now on we write $\mathcal{G}_p := \mathcal{G}_p(S)$. The next lemma is crucial as it allows us to find sections of $S$ containing the residue class of many elements of $S$ for many primes $p$.

\begin{lema}
\label{intermedio}
There exists a set $\mathcal{B} \subseteq S'$, $|\mathcal{B}| \gg |S|$, such that for every non empty section $\mathcal{B}_x$ of $\mathcal{B}$ there is a set of primes $P_x \hookrightarrow P' \hookrightarrow P$ with
$$ \left| \left\{ s \in S' : [s]_p \in \left[ \mathcal{B}_x \right]_p \right\} \right| \ge \frac{c|S|}{p},$$
for every $p \in P_x$, where $c>0$ does not depend on $x$ or $p$. 
\end{lema}

\begin{proof}
We begin by fixing a prime $p \in P'$ and considering some residue class $a \in \left[ \pi_1(\mathcal{G}_p) \right]_p$. Since $p$ is fixed we will simply write $\mathcal{G}_p(a)$ to denote those elements of $\mathcal{G}_p$ with first coordinate congruent to $a (\text{mod }p)$. Also, given a class ${\bf b} \in (\mathbb{Z}/p \mathbb{Z})^d$ we write $\mathcal{G}_p({\bf b})$ for those elements of $\mathcal{G}_p$ congruent to ${\bf b} (\text{mod }p)$. By the pigeonhole principle and the fact that by construction of $P'$ it is $|\left[ \mathcal{G}_p(a) \right]_p| \le Bp^{d-h-1}$ it follows that we may find some ${\bf b}_1 \in [\mathcal{G}_p(a)]_p \subseteq (\mathbb{Z}/p \mathbb{Z})^d$ with 
$$|\mathcal{G}_p({\bf b}_1)| \ge |\mathcal{G}_p(a)|/(Bp^{d-h-1}).$$
Consider now the set $\mathcal{B}_1 \subseteq \mathcal{G}_p(a)$ defined by
\begin{equation}
\label{B1}
 \mathcal{B}_1 := \bigcup_{s:[s]_p={\bf b}_1} \left( \mathcal{G}_p \right)_{\pi_1(s)},
 \end{equation}
that is, $\mathcal{B}_1$ is the union of those sections $(\mathcal{G}_p)_x$ in $\mathcal{G}_p$ containing a representative of ${\bf b}_1$. 

\medskip
Since each $(\mathcal{G}_p)_x$ is a $(C,d-h-1)$-generic set, we have that 
$$|(\mathcal{G}_p)_x| \ge \frac{p^{d-h-1}}{C} \left| (\mathcal{G}_p)_x({\bf b}_1) \right|$$
 and therefore
\begin{equation}
\label{Y_b}
|\mathcal{B}_1| \ge \frac{p^{d-h-1}}{C} \left| \mathcal{G}_p({\bf b}_1) \right| \ge \frac{1}{BC}|\mathcal{G}_p(a)|.
\end{equation}
Notice now that since $|\mathcal{G}_p(a)| \ge |\mathcal{B}_1|$ and $|[\mathcal{G}_p(a)]_p| \le Bp^{d-h-1}$, by the first inequality of (\ref{Y_b}) and the pigeonhole principle we may find another residue class ${\bf b}_2 \in [\mathcal{G}_p(a)]_p$ with 
\begin{align*}
|\mathcal{G}_p({\bf b}_2)| &\ge \frac{1}{Bp^{d-h-1}} |\mathcal{G}_p(a) \setminus \mathcal{G}_p({\bf b}_1)| \\
&\ge \frac{1}{Bp^{d-h-1}} \left( 1-\frac{C}{p^{d-h-1}} \right) |\mathcal{G}_p(a)|,
\end{align*}
which is at least $|\mathcal{G}_p(a)|/(2Bp^{d-h-1})$ if $p^{d-h-1} > 2C$. In such a case, if we now define $\mathcal{B}_2$ as in (\ref{B1}), but this time with respect to ${\bf b}_2$, the same reasoning that gives (\ref{Y_b}) implies $|\mathcal{B}_2| \ge \frac{1}{2BC}|\mathcal{G}_p(a)|$. Iterating this process we end up with a sequence ${\bf b} =\left\{ {\bf b}_1, \ldots, {\bf b}_q \right\}$ of residue classes, $q = \lceil{ \frac{p^{d-h-1}}{2C} }\rceil$, satisfying
\begin{align*}
|\mathcal{G}_p({\bf b}_j)| &\ge \frac{1}{Bp^{d-h-1}} \left|\mathcal{G}_p(a) \setminus \bigcup_{i=1}^{j-1} \mathcal{G}_p({\bf b}_i) \right| \\
&\ge \frac{1}{Bp^{d-h-1}} \left( 1-\frac{(q-1)C}{p^{d-h-1}} \right) |\mathcal{G}_p(a)| \\
&\ge \frac{|\mathcal{G}_p(a)|}{2Bp^{d-h-1}},
\end{align*}
and $|\mathcal{B}_j| \ge \frac{1}{2BC}|\mathcal{G}_p(a)|$. In particular, we have that
\begin{equation}
\label{multiplicity}
\sum_{j=1}^q |\mathcal{B}_j| \ge \frac{q}{2BC}|\mathcal{G}_p(a)|.
\end{equation}
Now, we consider the set
$$ \mathcal{B}[a] := \left\{ s \in \mathcal{G}_p(a) : \sum_{j=1}^q{\bf 1}_{s \in \mathcal{B}_j} \ge \frac{q}{4BC} \right\}.$$
Notice that $\mathcal{B}[a]_x := \mathcal{B}[a] \cap \pi_1^{-1}(x)$ equals $(\mathcal{G}_p)_x$ whenever this intersection is not empty. Also, (\ref{multiplicity}) implies 
\begin{equation}
\label{size B[a]}
\left|\mathcal{B}[a] \right| \ge \frac{1}{4BC}|\mathcal{G}_p(a)|.
\end{equation}
We see that $\mathcal{B}[a]$ is very close to what we want,  since if we take any nonempty section $\mathcal{B}[a]_x$ of this set, then there are at least $|\mathcal{G}_p(a)|/(4 B C)^2$ elements $s \in \mathcal{G}(a)$ such that $s \equiv y (\text{mod }p)$ for some $y \in \mathcal{B}[a]_x$.

\medskip
We now let $\mathcal{R} \subseteq [\pi_1(S)]_p$ consist of those residue classes $a \in \mathbb{Z}/p \mathbb{Z}$ with $|\mathcal{G}_p(a)| \ge \frac{1}{2p}|\mathcal{G}_p|$ and write
$$ \mathcal{B}[p] := \left\{ s \in S' : S_{\pi_1(s)}' \cap \mathcal{B}[a] \neq \emptyset \text{   for some }a \in \mathcal{R} \right\}.$$
In other words, $\mathcal{B}[p]$ consists of those sections of $S'$ intersecting $\bigcup_{a \in \mathcal{R}} \mathcal{B}[a]$. In particular, since $\mathcal{B}[p]$ contains the disjoint union $\bigcup_{a \in \mathcal{R}}\mathcal{B}[a]$, we see from (\ref{size B[a]}) and the definition of $\mathcal{R}$ that
$$ |\mathcal{B}[p]| \ge \frac{1}{8 B C} |\mathcal{G}_p| \ge c |S|,$$
for some constant $c$ independent of $p$.

\medskip
Recall now that $w(P') \ge c \log Q$. For an element $s \in S'$ write $P_s'$ for the set of primes $p \in P'$ for which $s \in \mathcal{B}[p]$. It follows from the above paragraph that for an appropriate choice of $c$ the set
\begin{equation}
\label{Y}
 \mathcal{B}:= \left\{ s \in S' : w(P_s') \ge c \log Q \right\},
 \end{equation}
satisfies $|\mathcal{B}| \ge c |S|$. It is easy to check that $\mathcal{B}$ is of the desired form.
\end{proof}

\medskip
To conclude the proof of Proposition \ref{control} we will show that if an $r$-polynomial vanishes at the sections $\mathcal{B}_x$ for $\gg_r 1$ distinct values of $x$, then it must also vanish at a positive proportion of $S$. To this end, we choose $m$ distinct sections of $S'$ having nontrivial intersection with $\mathcal{B}$, where $m=O_r(1)$ is to be specified later. Notice that by (\ref{assumption3}) and Lemma \ref{intermedio} this is always possible provided $N$ is sufficiently large. Call $\mathcal{L}= \left\{ S_{x_1}', \ldots, S_{x_m}' \right \}$ this set of sections. Let $P_{\mathcal{L}}$ consist of those primes $p$ for which there exists a pair of sections $S_{x_i}' \neq S_{x_j}'$ in $\mathcal{L}$ with $[S_{x_i}']_p \cap [S_{x_j}']_p \neq \emptyset$. Given such a pair of sections the fact that $[S_{x_i}']_p \cap [S_{x_j}']_p \neq \emptyset$ implies in particular that $x_i \equiv x_j (\text{mod }p)$. Since $x_i \neq x_j$ this implies that the sum of $\log p$ over such primes is bounded by $\log N$. Thus, we see that 
\begin{equation}
\label{P_L}
 \sum_{p \in P_{\mathcal{L}}} \log p \le {m \choose 2} \log N,
 \end{equation}
and this implies that $w(P_{\mathcal{L}}) \ll_r \log \log N$. 

\medskip
We now consider on $S'$ the function
$$\psi_{\mathcal{L}}(s) := \sum_{p \le Q} {\bf 1}_{\exists x \in \mathcal{L}:s \equiv x (\text{mod }p)} \log p.$$
Thus, $\psi_{\mathcal{L}}(s)$ measures the extent to which the residue classes occupied by $s$ have a representative in $\mathcal{L}$. If we write $P_i$ to denote the set of primes in Lemma \ref{intermedio} corresponding to the section $S_{x_i}' \cap \mathcal{B}$ of $\mathcal{B}$, it follows from this lemma and (\ref{P_L}) that
\begin{align*}
\sum_{s \in S'} \psi_{\mathcal{L}}(s) &\ge  \sum_{i=1}^m \sum_{p \in P_i \setminus P_{\mathcal{L}}}\sum_{s \in S'} {\bf 1}_{\exists x \in S_{x_i}' : s \equiv x (\text{mod }p)}\log p \\
&\ge \sum_{i=1}^m \sum_{p \in P_i \setminus P_{\mathcal{L}}} \frac{c|S|}{p}\log p \\
&\ge m|S| \left( c \log Q +O_r(\log \log N) \right)\\
&\ge c_0 m|S| \log Q,
\end{align*}
for some $c_0 > 0$ and sufficiently large $N$.

\medskip
Set $\delta = \frac{\varepsilon c_0}{4d}$ and suppose there are at most $\delta |S|$ elements $s \in S'$ with $\psi_{\mathcal{L}}(s) \ge 3r \log N$. Since $\psi_{\mathcal{L}}(s) \le m \log N$ for every $s \notin \mathcal{L}$ we conclude that
$$ c_0 m |S| \log Q \le |\mathcal{L}|2Q+ |S| 3r\log N +\delta |S| m \log N,$$
where we used that $\sum_{p\le Q}\log p \le 2Q$ for large $Q$. Hence, by (\ref{assumption3}) we derive that
$$ m \left( \frac{\varepsilon c_0}{2d} - \delta-\frac{4}{\log N} \right) \le 3r.$$
Taking $m=7 r/\delta$ we get a contradiction for sufficiently large $N$. We may therefore assume that the set
$$ A := \left\{ s \in S' : \psi_{\mathcal{L}}(s) \ge 3r \log N \right\},$$
has size $|A| \ge \delta |S|$ for the above choices of $m$ and $\delta$.

\medskip
We will now show that if an $r$-polynomial vanishes at $\mathcal{L}$, then it also vanishes at $A$. Indeed, let $f$ be such a polynomial and let $x \in A$ be arbitrary. By definition, we have $|f(x)|<N^{3r}$. On the other hand, if $p$ is a prime for which there exists some $y \in \mathcal{L}$ with $x \equiv y (\text{mod }p)$, then the fact that $f(y)=0$ implies that $p|f(x)$. But by definition of $A$ the product of all such $p$ is at least $N^{3r}$ so we see that the only way for this to hold is to have $f(x)=0$, which proves our claim.

\medskip
By the induction hypothesis and our construction of $S'$ we know that for each $S_{x_i}' \in \mathcal{L}$ we may find a $(r,1)$-characteristic set of size $O(r^{d-h-1})$. Taking the union of these $m$ sets we have thus found a set of size $O(r^{d-h})$ which is $(r,\delta)$-characteristic for $S$, with $\delta$ as above. This concludes the proof of Proposition \ref{control}.

\section{Ill-distributed sets with no algebraic structure}
\label{thin}

In this section we provide some examples of high dimensional ill-distributed sets possessing no algebraic structure. In particular, we show that the assertion of Theorem \ref{principal} fails when $\varepsilon = 0$. To begin with, we use a slight modification of the construction given in \cite[\S 4.3]{HV} to see that, given any $0 < \eta < 1$, one may construct a subset of $[N]$ of size $\gg (\log N)^{\eta}$ which occupies at most $p^{\eta}$ residue classes for every prime $p$ and which possesses no algebraic structure. Indeed, if $N$ is sufficiently large, we may find some integer $Q$ with $Q < \log N < 2Q$ such that the product of all primes $p \le Q$, say $R$, satisfies $N^{1/4} < R < N$ (this, of course, is very crude). For each prime $p \le Q$ choose $\lfloor p^{\eta} \rfloor$ residue classes. Then, by the Chinese remainder theorem, there are $\sim R^{\eta}$ elements below $R$ belonging to a selected class for every $p \le Q$. Choose $\lfloor (\log N)^{\eta}/2 \rfloor$ of these elements and call this set $X$. Notice that for all primes $p>Q$ we have $p^{\eta}>|X|$ and therefore $X$ occupies at most $p^{\eta}$ residue classes for these primes $p$. Since by construction it also occupies that many classes for all primes $p \le Q$, we get the claim.

\medskip
We now proceed to give some examples of ill-distributed sets with no algebraic structure. The first one already shows that Theorem \ref{principal} is best possible.

\begin{ej}
This follows readily from the above construction. Fix some pair of positive integers $d,h$ with $d \ge h+1$ and consider $h+1$ different sets $X_1, \ldots, X_{h+1}$ constructed as in the previous paragraph with $\eta=1/(h+1)$. If we define the set
$$ S := \left\{ (x_1, \ldots, x_d) \in [N]^d : x_i \in X_i \text{  } \forall 1 \le i \le h+1 \right\},$$
then we have that $|S| \gg N^{d-h-1} \log N$ while $|[S]_p| \le p^{d-h}$ for every prime $p$, from where it follows that we cannot take $\varepsilon = 0$ in Theorem \ref{principal}.
\end{ej}

\begin{ej}
\label{ej 2}
One can generalize the above example by ``perturbing'' arbitrary algebraic sets. We show a simple instance of this. Let $d=3$ and consider two polynomials $f,g \in \mathbb{Z}[x]$. Let $X$ and $Y$ be sets of size $\gg (\log N)^{1/2}$ occupying at most $p^{1/2}$ residue classes for every prime $p$. Then, we see that $$ \left\{ (x,f(x)\cdot X,g(x) \cdot Y ) : x \in [N] \right\}$$ is a big set of integer points occupying at most $p^2$ residue classes. 
\end{ej}

Finally, we show that not all counterexamples are perturbations of strongly algebraic sets.

\begin{ej}
Fix some small $\varepsilon > 0$. By the Chinese remainder theorem one can construct a set $X \subseteq [N]$ of size $|X| \sim N^{1-\varepsilon}$ occupying only one residue class for every prime $p \le \varepsilon \log N$. Take $K= \lfloor (\varepsilon \log N)^{1/3} \rfloor$ and let $f_1, \ldots, f_K,g_1, \ldots, g_K$ be a family of polynomials. Also, let $X_1, \ldots, X_K,Y_1, \ldots, Y_K$ be arbitrary sets of size at most $(\varepsilon \log N)^{1/3}$. Then
$$ \bigcup_{i=1}^K \left\{ (x,f_i(x) \cdot X_i,g_i(x) \cdot Y_i) : x \in X \right\}$$
is a big set of integer points occupying at most $p^2$ residue classes for every prime $p$. Notice that this construction is of a different nature than the one given in Example \ref{ej 2}, since the union of that many algebraic sets hardly retains any algebraic structure itself.
\end{ej}

It follows from the above examples that strange things can happen if one allows the set to possess too many very small sections. However, we shall show in Theorem \ref{regular teo} below that the methods of this paper do indeed work as long as one avoids this type of situations.

\section{Further results and conjectures}
\label{Further}

\subsection{A generalization of Theorem 1.1}
We now state the most general result which follows at once from the methods of this paper. To do this we need the concept of a $(k,\eta,\varepsilon,\rho)$-regular set. When $k=1$ we just take this to mean $|S| \ge N^{\varepsilon}$. Recursively, for any positive integer $k$, we say $S \subseteq [N]^d$ is $(k,\eta,\varepsilon,\rho)$-regular if for any $S' \subseteq S$ of size $|S'| \ge \eta |S|$ we can find a subset $S'' \subseteq S'$, $|S''| \ge \rho|S|$, satisfying 
\begin{equation*}
|\pi_i(A)| \ge N^{\varepsilon} \text{ for every }A \subseteq S'' \text{ with }|A| \ge |S''|/2
\end{equation*}
for some $1 \le i \le d$ and such that $S'' \cap \pi_i^{-1}(x)$, $i$ as before, is either empty or $(k-1,1,\varepsilon,\rho)$-regular for every $x \in [N]$. Although this definition seems complicated, one would expect a reasonable set to satisfy it for an appropriate choice of the parameters. As an example for which we shall give an application below, we notice that given any function $f:[N]^k \rightarrow [N^r]^t$ with $k,r,t$ arbitrary positive integers, the graph of $f$ is a $(k,\eta,1/2r,1/2)$-regular set provided $N$ is large in terms of $\eta$. This definition was chosen so that it satisfies all the assumptions made during the proof of Theorem \ref{principal}. Therefore, we immediately get the following generalization of this result.

\begin{teo}
\label{regular teo}
Let $k>0$ be some integer and $\eta, \varepsilon,\alpha, \rho > 0$ real numbers. Then there exists $C=O_{k,\eta,\varepsilon,\alpha,\rho}(1)$ such that for any $(k,\eta,\varepsilon,\rho)$-regular set $S$ occupying less than $\alpha p^k$ residue classes for every prime $p$ there exists a polynomial $f \in \mathbb{Z}[x_1, \ldots, x_d]$ of degree at most $C$ and coefficients bounded by $N^C$, such that $f$ vanishes at more than $(1-\eta) |S|$ points of $S$.
\end{teo}

It is important to note that one cannot hope to do much better than Theorem \ref{regular teo} in this generality, since the regularity conditions are necessary in order to avoid those constructions emerging from the Chinese Remainder Theorem as in \S \ref{thin}.

\subsection{Approximate reduction}
\label{Approximate reduction}
We shall give a quick application of Theorem \ref{regular teo} to the study of functions preserving some structure when reduced modulo a prime, that is, functions $f$ for which knowing the class of $x (\text{mod }p)$ gives us information about the class of $f(x) (\text{mod }p)$. Thus, given a positive integer $K$, we say a function $f:[N]^k \rightarrow [N^r]^t$ has $K$-approximate reduction if $\left| \left[f \left([N]^k({\bf a}) \right) \right]_p \right| \le K$ for every ${\bf a} \in \left( \mathbb{Z}/p \mathbb{Z} \right)^k$ and every prime $p$. When $K=1$ this implies the very strong property of \textit{recurrence} mod $p$ and using this, it was shown by Hall \cite{Hall} and Ruzsa \cite{Ruzsa} (see also \cite[\S XV.41]{handbook}) that for large $N$ the only functions having $1$-approximate reduction are polynomials (notice that we are assuming our functions to have polynomial growth, which is in fact a necessary condition \cite{Hall}). It follows from Theorem \ref{regular teo} that this is indeed a very robust phenomenon:

\begin{coro}
Suppose $f:[N]^k \rightarrow [N^r]^t$ has $K$-approximate reduction and let $\Gamma(f)$ be the graph of $f$. Then there exists $C=O_{k,r,t,K}(1)$ and a polynomial $P \in \mathbb{Z}[x_1, \ldots, x_d]$ of degree at most $C$ and coefficients bounded by $N^C$, such that $P$ vanishes at more than $(1-\eta)|\Gamma(f)|$ points of $\Gamma(f)$.
\end{coro}

\subsection{The Inverse Sieve Problem for $d=1$}
\label{d=1}
We conclude by mentioning a very strong version of the inverse sieve problem which is conjectured to hold for sets $S \subseteq [N]$ (see \cite[Problem 7.2]{CL} and \cite{HV}).

\begin{conj}
\label{uno}
Suppose $S \subseteq [N]$ is some set of integers of size $|S| \ge N^{\varepsilon}$ occupying less than $\alpha p$ residue classes for some $0 < \alpha < 1$ and every prime $p$. Then most of $S$ is contained in the image of an integer polynomial of degree bounded in terms of $\alpha$ and $\varepsilon$.
\end{conj}

As a more precise instance of this, they conjecture for example that if a set $S$ has size $|S| \ge N^{0.49}$ say, and occupies less than $2p/3$ residue classes mod $p$ for every prime $p$, then most of $S$ must be contained in a set of the form $\left\{ an^2+bn+c : n \in \mathbb{Z} \right\}$. This can be seen as an inverse conjecture for the large sieve \cite{Green,Montgomery}.

\medskip
Conjecture \ref{uno} seems to be hard. For example, it was shown by Green that if the residue classes occupied by $S$ lie outside some interval of length $(p-1)/2$ then $|S| \ll_{\varepsilon} N^{1/3+\varepsilon}$ for any $\varepsilon > 0$. However, even in this  particular case, to get a bound of the form $|S| \ll_{\varepsilon} N^{\varepsilon}$, it seems necessary to appeal to very deep conjectures of analytic number theory like the exponent pair conjecture (see \cite{Green}). Furthermore, as noted by Helfgott and Venkatesh \cite[\S 4.2]{HV}, Conjecture \ref{uno} implies that there are $ \ll_{\varepsilon} N^{\varepsilon}$ points on an irrational curve, which is itself a well known open problem.


\bigskip

\begin{thebibliography}{25}
\bibitem{BP} E. Bombieri and J. Pila, \emph{The number of integral points on arcs and ovals}, Duke Math. J. {\bf 59} (1989), 337-357.
\bibitem{CM} A. C. Cojocaru and M. R. Murty, \emph{Introduction to Sieve Methods and Their Applications}.
London Mathematical Society Student Texts, Cambridge University Press, 2005.
\bibitem{CL} E. Croot and V. F. Lev, \emph{Open problems in additive combinatorics}, in Additive Combinatorics, CRM Proc. Lecture Notes {\bf 43}, 207-233, Amer. Math. Soc., Providence, RI, 2007.
\bibitem{Green survey} B. J. Green, \emph{Approximate groups and their applications: work of Bourgain, Gamburd, Helfgott and Sarnak}, Current Events Bulletin of the AMS, 2010.
\bibitem{Green} B. J. Green, \emph{On a variant of the large sieve}, preprint, \href{http://arxiv.org/abs/0807.5037}{arXiv:0807.5037}.
\bibitem{GR} B. J. Green and I. Z. Ruzsa, \emph{Freiman's theorem in an arbitrary abelian group}, J. Lond. Math.
Soc. (2) {\bf 75} (2007), no. 1, 163-175.
\bibitem{GTZ} B. J. Green, T. Tao and T. Ziegler, \emph{An inverse theorem for the Gowers $U^{s+1}[N]$-norm}, preprint, \href{http://arxiv.org/abs/1009.3998}{arXiv:1009.3998}
\bibitem{Gallagher} P. X. Gallagher, \emph{A larger sieve}, Acta Arith. {\bf 18} (1971), 77-81.
\bibitem{Hall} R. R. Hall, \emph{On pseudopolynomials}, Mathematika {\bf 18} (1971), 71-77.
\bibitem{Helfgott} H. A. Helfgott, \emph{Growth and generation in $SL_2(\mathbb{Z}/p \mathbb{Z})$}, Ann. of Math. (2) {\bf 167} (2008), no. 2,
601-623.
\bibitem{HV} H. A. Helfgott and A. Venkatesh, \emph{How small must ill-distributed sets be?}, Analytic number theory. Essays in honour of Klaus Roth. Cambridge University Press, 2009, 224-234.
\bibitem{Kowalski} E. Kowalski, \emph{The ubiquity of surjective reduction in random groups}, notes available at \href{http://www.math.ethz.ch/~kowalski/notes-unpublished.html}{http://www.math.ethz.ch/~kowalski/notes-unpublished.html}.
\bibitem{Kowalski book} E. Kowalski, \emph{The large sieve and its applications: arithmetic geometry, random walks and discrete
groups}. Cambridge Tracts in Math. 175, Cambridge University Press, 2008.
\bibitem{Montgomery} H. L. Montgomery, \emph{A note on the large sieve}, J. London Math. Soc. {\bf 43} (1968), 93–98.
\bibitem{Ruzsa} I. Z. Ruzsa, \emph{On congruence preserving functions}, Mat. Lapok. {\bf 22} (1971), 125-134.
\bibitem{handbook} J. Sándor, D. S. Mitrinovic and B. Crstici, \emph{Handbook of number theory I}, Springer, 2006.
\bibitem{Tao} T. Tao, \emph{Freiman's theorem for solvable groups}, to appear in Cont. Disc. Math..
\bibitem{TV book} T. Tao and V. Vu, \emph{Additive combinatorics}, Cambridge Studies in Advanced Mathematics,
105. Cambridge University Press, Cambridge, 2006.
\bibitem{TV} T. Tao and V. Vu, \emph{Inverse Littlewood-Offord theorems and the condition number of random matrices}, Ann. of Math. (2) {\bf 169} (2009), no. 2, 595-632.
\end{thebibliography}
\end{document}